\documentclass[reqno,12pt,a4paper]{amsart}

\usepackage{tikz}

\usepackage[mathscr]{eucal}
\usepackage{graphics,epic}
\usepackage{amsfonts}
\usepackage{amscd}
\usepackage{latexsym}
\usepackage{amsmath,amssymb, amsthm, stmaryrd, bm, bbm}
\usepackage[all,2cell]{xy}

\newtheorem{theorem}{Theorem}[section]
\newtheorem*{theorem*}{Theorem}

\newtheorem{lemma}[theorem]{Lemma}
\newtheorem{proposition}[theorem]{Proposition}
\newtheorem{corollary}[theorem]{Corollary}

\newtheorem*{conjecture*}{Conjecture}

\newtheorem*{question*}{Question}
\theoremstyle{remark}
\newtheorem{remark}[theorem]{Remark}

\theoremstyle{definition}
\newtheorem{definition}[theorem]{Definition}

\newcommand{\dimv}{\underline{\dim}\,}

\numberwithin{equation}{section}

\begin{document}

\title[Riedtman's Lie algebra of $\Lambda(n-1,1,1)$]{On Riedtmann's Lie algebra of the gentle one-cycle algebra $\Lambda(n-1,1,1)$}

\author{Hui Chen, and Dong Yang}
\dedicatory{Dedicated to Jie Xiao on the occasion of his 60th birthday}

\address{Hui Chen, Department of Mathematics, Nanjing University, Nanjing, China;
School of Biomedical Engineering and Informatics, Nanjing Medical University, Nanjing, China.}
\email{huichen@njmu.edu.cn}

\address{Dong Yang, Department of Mathematics, Nanjing University, Nanjing 210093, P. R. China}
\email{yangdong@nju.edu.cn}

\begin{abstract} An extended version of Riedtmann's Lie algebra of the gentle one-cycle algebra $\Lambda(n-1,1,1)$ is computed and is shown to admit a Cartan decomposition by the positive roots of the root system of type $\mathbf{BC}_n$.

\end{abstract}
\maketitle

\section{Introduction}
\label{makereference:intro}

The gentle one-cycle algebra $\Lambda(n-1,1,1)$ is defined as the quotient of the path algebra of 
\begin{equation*}
\xymatrix@R=1ex{
1 \ar[r]&2\ar[r]&\cdots\ar[r]&n\ar@(ur,dr)^{\alpha}}
\end{equation*}
by the ideal generated by $\alpha^2$. This algebra is known to be derived-discrete \cite{Vossieck01,BobinskiGeissSkowronski04}, and in recent years it has attracted much attention mainly because it appears in the theory of cluster tubes \cite{Vatne11, Yang12, ZhouZhu14, FuGengLiu21,FuGengLiu20}. It is of particular interest that on the one hand due to \cite{BuanMarshVatne10}, the support $\tau$-tilting theory of this algebra can be used to model the cluster combinatorics of type $\mathbf{B}_n$, and on the other hand, a Caldero-Chapoton map is defined in \cite{ZhouZhu14,FuGengLiu21,FuGengLiu20} to realise the cluster algebra of type $\mathbf{C}_n$. Moreover, the algebra is Iwanaga--Gorenstein of dimension $1$ and is used to realise the positive part of the simple complex Lie algebra of type $\mathbf{C}_n$, as an example of the general theory on constructing non-simply-laced Lie algebras in \cite{GeissLeclercSchroeer16}.

This paper is concerned with the structure of Riedtmann's Lie algebra $L(A)$ of $A=\Lambda(n-1,1,1)$, which is the free $\mathbb{Z}$-module with basis the isomorphism classes of indecomposable $A$-modules endowed with the Lie bracket with structure constants given by the Euler characteristics of suitable varieties of submodules. The algebra $A$ has only finitely many isomorphism classes of indecomposable modules and a classification is given in \cite{BoosReineke12}. This helps us to obtain a complete description of the Lie bracket of $L(A)$ (Proposition~\ref{prop:lie-brackets-of-L(A)}). Moreover, we introduce a symmetric bilinear form on the Grothendieck group of $A$ in terms of the Cartan matrix of $A$ (Definition~\ref{def:bilinear-form}). It turns out to be a modified symmetric Euler form\footnote{Note that the usual Euler form is not well-defined since $A$ has infinite global dimension.} (Theorem~\ref{thm:euler-form-vs-cartan-matrix}).  This form allows us to establish a Gabriel's theorem for $A$ (Theorem~\ref{thm:Gabriel's-theorem}): taking dimension vectors defines a surjective map from the set of isomorphism classes of indecomposable $A$-modules to the set of positive roots of the root system of type $\mathbf{BC}_n$, which sends the symmetric Euler form to the natural bilinear form on roots. We also use this form to extend the Riedtmann Lie algebra $L(A)$ by adding an abelian Lie algebra given by the Grothendieck group of $A$, following \cite{Xiao97,PengXiao97}.

The next step is to describe the extended Riedtmann's Lie algebra $\tilde{L}(A)$ by generators and relations. We introduce a complex Lie algebra $\mathfrak{g}$, which is generated by $x_1, x_2, \ldots, x_{n-1}, x_n, x'_n, h_1, \ldots, h_n$ subject to the following relations:
\begin{itemize}
  \item[-] $\{x_1, \ldots, x_{n-1}, x_n, h_1, \ldots, h_{n-1}, h_n\}$ satisfy the Serre relations for the Borel subalgebra $\mathfrak{b}_{\mathbf{B}}$ of the simple complex Lie algebra of type $\mathbf{B}$;
  \item[-] $\{x_1, \ldots, x_{n-1}, x'_n, h_1, \ldots, h_{n-1}, \frac{1}{2}h_n\}$ satisfy the Serre relations for the Borel subalgebra $\mathfrak{b}_{\mathbf{C}}$ of the simple complex Lie algebra of type $\mathbf{C}$;
  \item[-] $[x_n, x'_n]=0$ and $[[x_{n-1},x_n],x'_n]=0$.
\end{itemize}
Then $\mathfrak{g}=\mathfrak{b}_{\mathbf{B}}+\mathfrak{b}_{\mathbf{C}}$. Moreover, the space $\mathfrak{h}=\mathrm{span}\{h_1,\ldots,h_n\}$ is a Cartan subalgebra of $\mathfrak{g}$ and the corresponding Cartan decomposition of $\mathfrak{g}$ is given by the positive roots of the root system of type $\mathbf{BC}_n$ (Theorem~\ref{thm:root-space-decomposition}).

Let $S_1,\ldots,S_n$ be the simple $A$-modules supported on the vertices $1,\ldots,n$, respectively, and let $S'_n$ be the unique 2-dimensional indecomposable $A$-module supported on the vertex $n$. Let $h_{S_1},\ldots,h_{S_n},h_{S'_n}$ be the image of $S_1,\ldots,S_n,S'_n$ in the Grothendieck group, respectively. The following theorem is our main result (see Theorem~\ref{thm:ringel-hall-lie-algebra}). 

\begin{theorem}
The assignment $h_i\mapsto h_{S_i} (1\leq i\leq n-1),\frac{1}{2}h_n\mapsto h_{S_n}, x_i\mapsto S_i (1\leq i\leq n), x'_n\mapsto S'_n$ extends to an isomorphism $\mathfrak{g}\longrightarrow\tilde{L}(A)\otimes_{\mathbb{Z}}\mathbb{C}$ of complex Lie algebras.
\end{theorem}

Therefore, the Lie subalgebra of $\tilde{L}(A)$ generated by $S_1,\ldots,S_{n-1},S_n,h_{S_1},\ldots,h_{S_{n-1}}$, $2h_{S_n}(=h_{S'_n})$ is isomorphic to $\mathfrak{b}_{\mathbf{B}}$ and  the Lie subalgebra generated by $S_1,\ldots,S_{n-1},S'_n$, $h_{S_1},\ldots,h_{S_{n-1}},\frac{1}{2}h_{S'_n}(=h_{S_n})$ is isomorphic to $\mathfrak{b}_{\mathbf{C}}$. We remark that the Lie subalgebra generated by $S_1,\ldots,S_{n-1},S'_n$ is exactly the Lie algebra $\mathcal{P}(\mathcal{M})$ in \cite{GeissLeclercSchroeer16}.

\medskip

The structure of the paper is as follows. In Section~\ref{s:Lie-algebra-type-BC} we recall the root system of type $\mathbf{BC}$, introduce the Lie algebra $\mathfrak{g}$ and study its Cartan decomposition. In Section~\ref{s:representation-theory} we introduce a symmetric bilinear form for the algebra $\Lambda(n-1,1,1)$ and establish a Gabriel's theorem. In Section~\ref{s:ringel-hall-lie-algebra} we describe the Lie bracket for the extended Riedtmann's Lie algebra $\tilde{L}(A)$ and prove the isomorphism between $\mathfrak{g}$ and $\tilde{L}(A)$. In the appendix we show that the symmetric bilinear form introduced in Section~\ref{s:representation-theory} is a modified symmetric Euler form.

\smallskip
\noindent{\it Acknowledgement.} The authors thank Bangming Deng and Changjian Fu for answering their questions. They thank Changjian Fu for reading a preliminary version. The second author acknowledges support by the National Natural Science Foundation of China No.12031007.

\section{A Lie algebra of type $\mathbf{BC}^+$}
\label{s:Lie-algebra-type-BC}

In this section we introduce a complex Lie algebra, which admits a Cartan decomposition by the positive roots of the root system of type $\mathbf{BC}$.

\subsection{The root system of type $\mathbf{BC}$}
\label{ss:root-sysmtem-type-bc}

Let $\mathbb{R}^n$ be the $n$-dimensional real space with the usual metric $(-,-)$, and $\varepsilon_1,\ldots, \varepsilon_n$ be the natural basis. In classical Lie theory, the root systems of type $\mathbf{B_n}$, $\mathbf{C_n}$ and $\mathbf{BC_n}$ are (see Bourbaki \cite[Chapter VI, Sections 4.5, 4.6 and 4.14]{Bourbaki68})
\begin{align*}
\Phi_\mathbf{B}=&\{\pm\varepsilon_i\pm\varepsilon_j\mid 1\leq i< j \leq n\} \cup \{\pm \varepsilon_i \mid 1\leq i\leq n\},\\
\Phi_\mathbf{C}=&\{\pm\varepsilon_i\pm\varepsilon_j\mid 1\leq i< j \leq n\} \cup \{\pm 2\varepsilon_i \mid 1\leq i\leq n\},\\
\Phi_\mathbf{BC}=&\{\pm\varepsilon_i\pm\varepsilon_j\mid 1\leq i< j \leq n\} \cup \{\pm \varepsilon_i \mid 1\leq i\leq n\}\cup \{\pm 2\varepsilon_i \mid i=1,\ldots,n\}.
\end{align*}
Clearly $\Phi_\mathbf{BC}=\Phi_\mathbf{B}\cup \Phi_\mathbf{C}$. The corresponding simple roots are
\begin{align*}
\Delta_\mathbf{B}=&\{\varepsilon_i-\varepsilon_{i+1}\mid 1\leq i\leq n-1\} \cup \{\varepsilon_n \},\\
\Delta_\mathbf{C}=&\{\varepsilon_i-\varepsilon_{i+1}\mid 1\leq i\leq n-1\} \cup \{2\varepsilon_n\},\\
\Delta_\mathbf{BC}=&\{\varepsilon_i-\varepsilon_{i+1}\mid 1\leq i\leq n-1\} \cup \{\varepsilon_n \}.
\end{align*}
The corresponding positive roots are
\begin{align*}
\Phi^+_\mathbf{B}=&\{\varepsilon_i+\varepsilon_j\mid 1\leq i< j \leq n\} \cup\{\varepsilon_i-\varepsilon_j\mid 1\leq i< j \leq n\} \cup \{\varepsilon_i \mid 1\leq i\leq n\},\\
\Phi^+_\mathbf{C}=&\{\varepsilon_i+\varepsilon_j\mid 1\leq i< j \leq n\} \cup\{\varepsilon_i-\varepsilon_j\mid 1\leq i< j \leq n\} \cup \{2\varepsilon_i \mid 1\leq i\leq n\},\\
\Phi^+_\mathbf{BC}=&\{\varepsilon_i\pm\varepsilon_j\mid 1\leq i< j \leq n\}  \cup \{\varepsilon_i \mid 1\leq i\leq n\} \cup\{2\varepsilon_i \mid 1\leq i\leq n\}.
\end{align*}

In the rest of this paper we will use $\Phi$, $\Phi^{+}$ and $\Delta$ to denote $\Phi_\mathbf{BC}$, $\Phi^{+}_\mathbf{BC}$ and $\Delta_\mathbf{BC}$ for simplicity.

\subsection{The Borel subalgebras of the Lie algebras of type $\mathbf{B}$ and type $\mathbf{C}$}
\label{ss:Lie-alg-type-B-and-C}
In this section, we review some basic results on the Borel subalgebras of the simple complex Lie algebras of type $\mathbf{B_n}$ and type $\mathbf{C_n}$. Most of the results in this subsection are from J. E. Humphreys \cite{Humphreys72} and V. Kac \cite{Kac90}, but we will change some notations.

\subsubsection{The Borel subalgebra of the Lie algebra of type $\mathbf{B}$}
\label{B}\label{sss:lie-algebra-type-B}

Let $\mathfrak{b}_{\mathbf B}$ be the Borel subalgebra of the simple complex Lie algebra of type $\mathbf{B}_n$, i.e., it is the complex Lie algebra generated by $\{x_i, h_i\mid 1\leq i\leq n\}$ with relations:
\begin{itemize}\label{RL}
\item[\rm(B1)] $[h_i, h_j]=0$, $1\leq i,j \leq n$.
\item[\rm(B2)] \[[h_i, x_j]=\left\{
                              \begin{array}{ll}
                                2x_j, & 1\leq i=j\leq n; \\
                                -x_j, &\mid i-j\mid=1 \text{ but } i\neq n;\\
                                -2x_j, & j=n-1 \text{ and } i=n;\\
                                0,& \text{otherwise}.
                              \end{array}
                            \right.\]
\item[\rm(B3)] When $i\neq j$, \[\left\{
                              \begin{array}{ll}
                              [x_i, [x_i, x_j]]=0, &\mid i-j\mid=1 \text{ but } i\neq n;\\
                              {[}x_i,{[}x_i,{[}x_i,x_j{]]]}=0, & j=n-1 \text{ and } i=n;\\
                              {[}x_i,x_j{]}=0,& \text{otherwise}.
                              \end{array}
                            \right.
\]
\end{itemize}

Let $\mathfrak{h}$ be the subspace of $\mathfrak{b}_{\mathbf B}$ spanned by $h_1, \ldots, h_n$ and $\mathfrak{g}_{\mathbf B}^+$ be the Lie subalgebra generated by $x_1,\ldots,x_n$. Then $\mathfrak{h}$ is a maximal abelian subalgebra of $\mathfrak{g}_{\mathbf B}$, $\mathfrak{b}_{\mathbf B}=\mathfrak{h}\oplus\mathfrak{g}_{\mathbf B}^+$ and $\mathfrak{g}_{\mathbf B}^+$ admits the decomposition $\mathfrak{g}_{\mathbf B}^+=\bigoplus_{\mu\in \Phi^+_{\mathbf{B}}}\mathfrak{g}_{\mathbf B}(\mu)$, where
\[\mathfrak{g}_{\mathbf B}(\mu)=\{x\in \mathfrak{b}_{\mathbf B}\mid [h, x]=\mu(h)x, \forall h\in \mathfrak{h}\},\]
and the matrix $(\varepsilon_i(h_j))_{i,j}$ is
\[\left(
 \begin{array}{ccccc}
  1 &  &  &    &  \\
  -1 & 1 & \ &  &  \\
   & \ddots & \ddots &  &  \\
   &  & -1  & 1 &  \\
   &  &    & -1& 2 \\
  \end{array}
  \right).\]
\noindent Precisely, putting $x_{i,i}=x_i$ for $1\leq i\leq n$ and $x_{i,j}=[x_{i,j-1}, x_j]$ for $1\leq i<j\leq n$, we have:
\begin{enumerate}
  \item $\mathfrak{g}_{\mathbf B}(\varepsilon_i-\varepsilon_j) (1\leq i< j \leq n)$ is $1$-dimensional with basis $x_{i,j-1}$.
  \item $\mathfrak{g}_{\mathbf B}(\varepsilon_i) (1\leq i\leq n)$ is $1$-dimensional with basis $x_{i,n}$.
  \item $\mathfrak{g}_{\mathbf B}(\varepsilon_i+\varepsilon_j) (1\leq i< j \leq n)$ is $1$-dimensional with basis $[x_{i,n},x_{j,n}]$.
\end{enumerate}

\subsubsection{The Borel subalgebra of the Lie algebra of type $\mathbf{C}$} 
\label{C}\label{sss:lie-algebra-type-C}
Let $\mathfrak{b}_{\mathbf C}$ be the Borel subalgebra of the simple complex Lie algebra of type $\mathbf{C}_n$, i.e., it is the complex Lie algebra generated by $\{x'_i, h'_i\mid 1\leq i\leq n\}$ with relations:
\begin{itemize}\label{RL}
  \item[\rm(C1)] $[h'_i, h'_j]=0$, $1\leq i,j \leq n$.
  \item[\rm(C2)]  \[[h'_i, x'_j]=\left\{
                              \begin{array}{ll}
                                2x'_j, & 1\leq i=j\leq n; \\
                                -x'_j, &\mid i-j\mid=1 \text{ but } j\neq n;\\
                                -2x'_j, & j=n \text{ and } i=n-1;\\
                                0,& \text{otherwise}.
                              \end{array}
                            \right.
\]
  \item[\rm(C3)] When $i\neq j$, \[\left\{
                              \begin{array}{ll}
                              [x'_i,[x'_i,x'_j]]=0, &\mid i-j\mid=1 \text{ but } j\neq n;\\
                              {[}x'_i,{[}x'_i,{[}x'_i,x'_j{]]]}=0, & j=n \text{ and } i=n-1;\\
                              {[}x'_i,x'_j{]}=0,& \text{otherwise}.
                              \end{array}
                            \right.
\]
\end{itemize}

Let $\mathfrak{h}'$ be the subspace of $\mathfrak{b}_{\mathbf C}$ spanned by $h'_1, \ldots, h'_n$ and $\mathfrak{g}_{\mathbf C}^+$ be the Lie subalgebra generated by $x'_1,\ldots,x'_n$. Then $\mathfrak{h}'$ is a maximal abelian subalgebra of $\mathfrak{b}_{\mathbf C}$, $\mathfrak{b}_{\mathbf C}=\mathfrak{h}\oplus\mathfrak{g}_{\mathbf C}^+$ and $\mathfrak{g}_{\mathbf C}^+$ admits the decomposition $\mathfrak{g}_{\mathbf C}^+=\bigoplus_{\mu\in \Phi_{\mathbf{C}}^+}\mathfrak{g}_{\mathbf C}(\mu)$, where \[\mathfrak{g}_{\mathbf C}(\mu)=\{x\in \mathfrak{b}_{\mathbf C}\mid [h', x]=\mu(h')x, \forall h'\in \mathfrak{h'}\},\]
and the matrix $(\varepsilon_i(h'_j))_{i,j}$ is
\[\left(
 \begin{array}{ccccc}
  1 &  &  &    &  \\
  -1 & 1 & \ &  &  \\
   & \ddots & \ddots &  &  \\
   &  & -1  & 1 &  \\
   &  &    & -1& 1\\
  \end{array}
  \right).\]

\noindent Precisely, putting $x'_{i,i}=x'_i$ for $1\leq i\leq n$ and $x'_{i,j}=[x'_{i,j-1}, x'_j]$ for $1\leq i<j\leq n$, we have
\begin{enumerate}
  \item $\mathfrak{g}_{\mathbf C}(\varepsilon_i-\varepsilon_j)(1\leq i< j \leq n)$ is $1$-dimensional with basis $x'_{i,j-1}$.
  \item $\mathfrak{g}_{\mathbf C}(2\varepsilon_i) (1\leq i< n)$ is $1$-dimensional with basis $[x'_{i,n-1},x'_{i,n}]$ and $\mathfrak{g}_{\mathbf C}(2\varepsilon_n)$ is $1$-dimensional with basis $x'_{n,n}$.
  \item $\mathfrak{g}_{\mathbf C}(\varepsilon_i+\varepsilon_j) (1\leq i< j < n)$ is $1$-dimensional with basis $[x'_{j,n-1},x'_{i,n}]$ and $\mathfrak{g}_{\mathbf C}(\varepsilon_i+\varepsilon_n) (1\leq i< n)$ is $1$-dimensional with basis $x'_{i,n}$.
\end{enumerate}

\subsection{A Lie algebra of type $\mathbf {BC}^+$}
\label{ss:lie-algebra-type-bc}

\begin{definition}
Let $\mathfrak{g}$ be the complex Lie algebra generated by \[x_1, x_2, \ldots, x_{n-1}, x_n, x'_n, \text{ and }h_1, \ldots, h_n\] satisfying the following relations:
\begin{enumerate}
  \item[(BC1)] $\{x_1, \ldots, x_n, h_1, \ldots, h_n\}$ satisfy the relations (B1), (B2) and (B3) as in Section~\ref{sss:lie-algebra-type-B};
  \item[(BC2)] $\{x'_1, \ldots, x'_n, h'_1, \ldots, h'_n\}$ satisfy the relations (C1), (C2) and (C3) as in Section~\ref{sss:lie-algebra-type-C}, where $x'_i=x_i$, and $h'_i=h_i$ for $1\leq i\leq n-1$, and $h'_n=\frac{1}{2}h_n$;
  \item[(BC3)] $[x_n, x'_n]=0$ and $[[x_{n-1},x_n],x'_n]=0$.
\end{enumerate}
\end{definition}

It is clear that $\mathfrak{b}_{\mathbf{B}}$ and $\mathfrak{b}_{\mathbf{C}}$ are Lie subalgebras of $\mathfrak{g}$, and that $\mathfrak{h}={\rm span} \{h_1, \ldots, h_n\}$ is a maximal abelian subalgebra of $\mathfrak{g}$. Denote by $\mathfrak{g}^+$ the Lie subalgebra of $\mathfrak{g}$ generated by $x_1, \ldots,  x_n, x'_n$. Then $\mathfrak{g}_{\mathbf B}^+$ and $\mathfrak{g}_{\mathbf C}^+$ are Lie subalgebras of $\mathfrak{g}^+$. We keep the notation in Sections~\ref{sss:lie-algebra-type-C} and~\ref{sss:lie-algebra-type-B} for basis elements of $\mathfrak{g}_{\mathbf B}^+$ and $\mathfrak{g}_{\mathbf C}^+$. Notice that $x'_{i,j}=x_{i,j}$ for $1\leq i\leq j\leq n-1$.

\begin{lemma}
\label{BC}\label{lem:BC=B+C}
$\mathfrak{g}=\mathfrak{h}\oplus\mathfrak{g}^+$, 
and $\mathfrak{g^+}=\mathfrak{g}_{\mathbf B}^++\mathfrak{g}_{\mathbf C}^+$.
\end{lemma}
\begin{proof}
We need to  prove $[\mathfrak{g}_{\mathbf B}^+, x'_n]\subseteq\mathfrak{g}_{\mathbf B}^++\mathfrak{g}_{\mathbf C}^+$ and $[\mathfrak{g}_{\mathbf C}^+, x_n]\subseteq\mathfrak{g}_{\mathbf B}^++\mathfrak{g}_{\mathbf C}^+$. We only show the first inclusion since the second one is similar. Precisely, we show $[x, x'_n]\in\mathfrak{g}_{\rm C}^+$ for basis elements $x$ of $\mathfrak{g}_{\mathbf B}^+$.

Case 1: $x=x_{i,j-1}$ with $1\leq i< j \leq n$. The element $[x_{i,j-1}, x_n']$ is generated by $x_1, \ldots, x_{n-1}, x'_{n}$, and hence belongs to $\mathfrak{g}_{\mathbf C}^+$.

\smallskip
Case 2: $x=x_{i,n}$ with $1\leq i\leq n$. We claim that $[x,x'_n]=0$. For $i=n-1$ and $i=n$ this is the condition (BC3). For $1\leq i\leq n-2$ we have
\begin{align*}
  [x_{i,n}, x_n']&=[[[x_{i,n-2},x_{n-1}],x_n],x_n']\\
&=[[x_{i,n-2},[x_{n-1},x_n],x_n']]+[[x_{n-1},[x_n,x_{i,n-2}]],x_n']\\
&=[x_{i,n-2},[[x_{n-1},x_n],x_n']]+[[x_{n-1},x_n],[x_n',x_{i,n-2}]]\\
&=0.
\end{align*}
For the third equality notice that $[x_n,x_{i,n-2}]\in\mathfrak{g}_{\mathbf B}(\varepsilon_n+(\varepsilon_i-\varepsilon_{n-1}))=0$. The last equality holds because (BC3) holds and $[x_n',x_{i,n-2}]\in\mathfrak{g}_{\mathbf C}(2\varepsilon_n+(\varepsilon_i-\varepsilon_{n-1}))=0$.

\smallskip
Case 3: $x=[x_{i,n},x_{j,n}]$ with $1\leq i< j \leq n$.  Then $[x,x'_n]=0$, due to Case 2.
\end{proof}

The following result is a consequence of Lemma~\ref{lem:BC=B+C}.
\begin{proposition}
\label{prop:basis}
The Lie algebra $\mathfrak{g}$ has a basis
\begin{itemize}
    \item[\rm (g1)] $x_{i,j}(1\leq i\leq j\leq n), [x_{i,n}, x_{j,n}](1\leq i<j\leq n), x'_{i,n}(1\leq i\leq n),  [x_{j,n-1},x'_{i,n}](1\leq i\leq j\leq n-1);$
    \item[\rm (g2)] $h_1, \ldots, h_n.$
  \end{itemize}
Consequently, the dimension of $\mathfrak{g}$ is $\frac{3n^2+3n}{2}$, and $\mathfrak{g}=\mathfrak{b}_B+\mathfrak{b}_C$.
\end{proposition}

It follows from Lemma~\ref{lem:BC=B+C} that $\mathfrak{h}$ is a Cartan subalgebra of $\mathfrak{g}$. Then by Proposition~\ref{prop:basis} and the description of the root spaces in Section~\ref{ss:Lie-alg-type-B-and-C}, we obtain the following Cartan decomposition of $\mathfrak{g}$.
\begin{theorem}
\label{thm:root-space-decomposition} 
$\mathfrak{g}$ has the decomposition
$\mathfrak{g}=\bigoplus_{\mu\in\Phi^+\cup\{0\}}\mathfrak{g}(\mu)$, where $\mathfrak{g}(\mu)=\{x\in \mathfrak{g}\mid [h, x]=\mu(h)x, \forall h\in \mathfrak{h}\}$
satisfies
\begin{itemize}
  \item $\mathfrak{h}=\mathfrak{g}(0)$,
  \item  $\mathfrak{g}(\varepsilon_i-\varepsilon_j) (1\leq i< j \leq n)$ is $1$-dimensional with basis $x_{i,j-1}$;
  \item $\mathfrak{g}(\varepsilon_i) (1\leq i\leq n)$ is $1$-dimensional with basis $x_{i,n}$;
  \item $\mathfrak{g}(2\varepsilon_i) (1\leq i< n)$ is $1$-dimensional with basis $[x_{i,n-1},x'_{i,n}]$ and $\mathfrak{g}(2\varepsilon_n)$ is $1$-dimensional with basis $x'_{n,n}$;
  \item $\mathfrak{g}(\varepsilon_i+\varepsilon_j) (1\leq i< j < n)$ is $2$-dimensional with basis $[x_{i,n},x_{j,n}]$ and $[x_{j,n-1},x'_{i,n}]$, and $\mathfrak{g}(\varepsilon_i+\varepsilon_n) (1\leq i < n)$ is $2$-dimensional with basis $[x_{i,n},x_n]$ and $x'_{i,n}$;
\end{itemize}
\end{theorem}

\begin{corollary}
\label{cor:type-b-and-c-as-quotient}
The quotient of $\mathfrak{g}$ by the Lie ideal generated by $x'_n$ (respectively, by $x_n$) is isomorphic to $\mathfrak{b}_{\mathbf{B}}$ (respectively, $\mathfrak{b}_{\mathbf{C}}$).
\end{corollary}
\begin{proof}
By the proof of Lemma~\ref{lem:BC=B+C}, the Lie ideal generated by $x'_n$ has a basis $x'_{i,n}~(1\leq i\leq n)$, $[x_{j,n-1},x'_{i,n}]~(1\leq i\leq j\leq n-1)$. Similarly, the Lie ideal generated by $x_n$ has a basis $x_{i,n}~(1\leq i\leq n)$, $[x_{i,n}, x_{j,n}]~(1\leq i<j\leq n)$. The desired result follows immediately.
\end{proof}

\section{A Gabriel's theorem for the algebra $\Lambda(n-1,1,1)$}
\label{s:representation-theory}

In this section we establish a Gabriel's theorem for the gentle one-cycle algebra $\Lambda(n-1,1,1)$.

\medskip
Let $K=\mathbb{C}$. Consider the bound quiver $(Q,I)$
\begin{equation*}
\xymatrix@R=1ex{
1 \ar[r]&2\ar[r]&\cdots\ar[r]&n\ar@(ur,dr)^{\alpha}},\qquad \alpha^2.
\end{equation*}
The path algebra of this bound quiver is exactly the derived-discrete algebra $\Lambda(n-1,1,1)$ in the notation of \cite{BobinskiGeissSkowronski04}. In the rest of this paper we will denote this algebra by $A$ and we will identity an $A$-module with a representation of $(Q,I)$, i.e., a representation of $Q$ satisfying the relation $\alpha^2=0$. Let ${\rm rep}_K(Q,I)$ denote the category of finite-dimensional representations of $(Q,I)$.
Consider the following representations:
\begin{itemize}
  \item $U_{i,j}$ for $1 \leq j\leq i \leq n$:\\
\[\xymatrix@R=1ex{
0 \ar[r]^0&\cdots\ar[r]^0&0 \ar[r]^0&K\ar[r]^1&\cdots\ar[r]^1&K \ar[r]^{e_1}&K^2\ar[r]^{\rm I}&\ldots\ar[r]^{\rm I}&K^2\ar@(ur,dr)^{\alpha}\\
&&&\bullet\raisebox{-4mm}{\hspace{-2mm}$j$}\ar[r]&\cdots\ar[r]&\bullet \ar[r]&\bullet\ar[r] \raisebox{-4mm}{\hspace{-2mm}$i$}&\cdots\ar[r]&\bullet\ar@(r,r)[d]\raisebox{-4mm}{\hspace{-2mm}$n$}\\
&&&&&&\bullet\ar[r]&\cdots\ar[r]&\bullet}\\\]
  \item $U_{i,j}$ for $1 \leq i < j \leq n$:\\
\[\xymatrix@R=1ex{
0 \ar[r]^0&\cdots\ar[r]^0&0 \ar[r]^0&K\ar[r]^1&\ldots\ar[r]^1&K \ar[r]^{e_2}&K^2\ar[r]^{\rm I}&\ldots\ar[r]^{\rm I}&K^2\ar@(ur,dr)^{\alpha}\\
&&&&&&\bullet\ar[r]&\cdots\ar[r]&\bullet\ar@(r,r)[d]\\
&&&\bullet\raisebox{4mm}{\hspace{-2mm}$i$}\ar[r]&\ldots\ar[r]&\bullet \ar[r]&\bullet\ar[r] \raisebox{4mm}{\hspace{-2mm}$j$}&\cdots\ar[r]&\bullet\raisebox{4mm}{\hspace{-2mm}$n$}}\\\]
  \item $V_i$ for $1 \leq i \leq n$:\\
\[\xymatrix@R=1ex{
0 \ar[r]^0&\cdots\ar[r]^0&0 \ar[r]^0&K\ar[r]^1&\cdots\ar[r]^1&K\ar@(ur,dr)^{0}\\
&&&\bullet\raisebox{4mm}{\hspace{-2mm}$i$}\ar[r]&\cdots\ar[r]&\bullet\raisebox{4mm}{\hspace{-2mm}$n$}}\\\]
  \item $W_{i,j}$ for $1 \leq i \leq j <n$:\\
\[\xymatrix@R=1ex{
0 \ar[r]^0&\cdots\ar[r]^0&0 \ar[r]^0&K\ar[r]^1&\cdots\ar[r]^1&K\ar[r]^{0}&0\ar[r]^{0}&\cdots\ar[r]^{0}&0\ar@(ur,dr)^{0}\\
&&&\bullet\raisebox{4mm}{\hspace{-2mm}$i$}\ar[r]&\cdots\ar[r]&\bullet\raisebox{4mm}{\hspace{-2mm}$j$}}.\\\]
\end{itemize}
\noindent Here, $e_1 =\left(
              \begin{array}{c}
                1 \\
                0 \\
              \end{array}
            \right)$
, $e_2 =\left(
              \begin{array}{c}
                0 \\
                1 \\
              \end{array}
            \right)$, $\rm{I}=\left(
                   \begin{array}{cc}
                     1 & 0 \\
                     0 & 1 \\
                   \end{array}
                 \right)$, and $\alpha=\left(
              \begin{array}{cc}
                0 & 0 \\
                1 & 0 \\
              \end{array}
            \right)
$.

\begin{theorem}
[{\cite[Theorem 3.3]{BoosReineke12}}]
The representations $U_{i,j}$, $V_i$ and $W_{i,j}$ form a complete set of representatives of the isoclasses of indecomposable objects in ${\rm rep}_K(Q,I)$.
\end{theorem}\label{repn}

The simple representations in ${\rm rep}_K(Q,I)$ are $\{S_i=W_{i,i}|1 \leq i \leq n-1\}\cup \{S_n=V_{n}\}$ and $P_1=U_{n,1},\ldots, P_n=U_{n,n}$ form a complete set of pairwise non-isomorphic indecomposible projective representations of ${\rm rep}_K(Q,I)$. Thus the Cartan matrix of the algebra $A$ in the sense of \cite[Defintion 3.7]{AssemSimsonSkowronski06} is the $n \times n$-matrix
\[C_A=\left(
 \begin{array}{ccccc}
  1 & 0 & 0 & \ldots   & 0 \\
  1 & 1 & \ddots & \ddots & \vdots \\
  \vdots & \vdots & \ddots & \ddots   & 0 \\
  1 & 1 & \ldots   & 1 & 0 \\
  2 & 2 &  \ldots  & 2 & 2 \\
  \end{array}
  \right).\]

Recall that \textbf{the dimension vector} of a representation $M$ of the quiver $Q$ is defined to be the vector ${\rm \underline{dim}} M=({\rm dim}_KM_1, \ldots, {\rm dim}_KM_n).$
Below we define a symmetric bilinear form on the Grothendieck group $K_0(\mathrm{rep}_K(Q,I))$ with integer values.. In the appendix we will explain that it is a modified symmetric Euler form.
\begin{definition}
\label{def:bilinear-form}
For any $M$ and $N$ in ${\rm rep}_K(Q,I)$, define 
\[
(M, N)_A=({\rm \underline{dim}} M)(C_A^{-1}+C_A^{-t})({\rm \underline{dim}} N)^t.
\]
\end{definition}

Then we obtain a similar result as Gabriel's theorem \cite{Gabriel72}:
\begin{theorem}\label{Bij}
\label{thm:Gabriel's-theorem}
\begin{itemize}
  \item[\rm (a)] Let $D$ be the set of dimension vectors of indecomposible representations in ${\rm rep}_K(Q,I)$, and let $\alpha_i={\rm \underline{dim}} S_i$. There is a bijection: $D\longrightarrow \Phi^+$ which takes $\alpha_i$ to $\varepsilon_i-\varepsilon_{i+1} (1\leq i<n)$, $\alpha_n$ to $\varepsilon_n$ and which preserves addition. From now on, we identify $D$ with $\Phi^+$.
  \item[\rm (b)] The fiber of $\alpha\in \Phi^+$ under $\rm \underline{dim}$ is:
  \[\left\{
     \begin{array}{ll}
       W_{i,j-1}, & \hbox{if $\alpha=\varepsilon_i-\varepsilon_j~(1\leq i<j\leq n)$;} \\
       V_i, & \hbox{if $\alpha=\varepsilon_i~(1\leq i\leq n)$;} \\
       U_{i,i}, & \hbox{if $\alpha=2\varepsilon_i~(1\leq i\leq n)$;} \\
       U_{i,j}\text{ and } U_{j,i}, & \hbox{if $\alpha=\varepsilon_i+\varepsilon_j~(1\leq i<j\leq n)$.}
     \end{array}
   \right.\]
  \item[\rm (c)] For any $M$ and $N$ in ${\rm rep}_K(Q,I)$, we have $( M,  N)_A=({\rm \underline{dim}}M  , {\rm \underline{dim}} N)$.
\end{itemize}
\end{theorem}
\begin{proof}
We only need to prove (c). This is because the matrix
\[C_A^{-1}+C_A^{-t}=\left(
 \begin{array}{ccccc}
  2 & -1 &  &    &  \\
  -1 & 2 & \ddots &  &  \\
   & \ddots & \ddots & \ddots   &  \\
   &  & \ddots   & 2 & -1 \\
   &  &    & -1 & 1 \\
  \end{array}
  \right)\]
is exactly the matrix of the bilinear form $(-,-)$ with respect to the basis $\Delta=\{\alpha_1,\ldots,\alpha_n\}$.
\end{proof}

\section{Riedtmann's Lie algebra for $\Lambda(n-1,1,1)$}
\label{s:ringel-hall-lie-algebra}

In this section we compute an extended version of Riedtmann's Lie algebra of $\Lambda(n-1,1,1)$ and show that it is isomorphic to the Lie algebra $\mathfrak{g}$ introduced in Section~\ref{ss:lie-algebra-type-bc}.

\subsection{The extended Riedtmann's Lie algebra}

We adopt the notation in Section~\ref{s:representation-theory}.
Let $\mathcal{P}(A)$ be a set of representatives of all isomorphism classes of objects of ${\rm rep}_K(Q,I)$. Let $\mathcal{H}(A)$ be the free $\mathbb{Z}$-module with basis $\{u_X\mid X\in \mathcal{P}(A)\}$. According to \cite{Riedtmann94} (see also \cite{Lusztig91b, Schofieldb}), the following multiplication makes $\mathcal{H}(A)$ a $\mathbb{Z}$-algebra with identity $u_{0}$:
\[u_X\cdot u_Y=\sum_{Z\in \mathcal{P}(A)}\chi(V(X,Y;Z))u_Z,\]
where $V(X,Y;Z)=\{0\subseteq Z_1 \subseteq Z\mid Z_1 \cong Y, Z/Z_1 \cong X\}$\footnote{This is opposite to Riedtmann's multiplication.} and $\chi(V(X,Y;Z))$ is the Euler characteristic of $V(X,Y;Z)$. This sum is finite because $\chi(V(X,Y;Z))\neq 0$ implies that there is a short exact sequence $0\rightarrow X\rightarrow Z\rightarrow Y\rightarrow 0$, and hence $\underline{\rm dim}Z=\underline{\rm dim}X+\underline{\rm dim}Y$.

Let $\mathcal{I}(A)=\{U_{i,j}\mid 1\leq j\leq i\leq n\}\cup \{U_{i,j}\mid 1\leq i< j\leq n\}\cup \{V_{i}\mid 1\leq i\leq n\}\cup \{W_{i,j}\mid 1\leq i\leq j< n\}$, and let $L(A)$ be the $\mathbb{Z}$-submodule of $\mathcal{H}(A)$ spanned by $\{u_X\mid X\in \mathcal{I}(A)\}$. Then according to \cite[Corollary 2.3]{Riedtmann94} , $L(A)$ is a Lie subalgebra of $\mathcal{H}(A)$ with respect to the commutator $[u_X, u_Y]=u_X\cdot u_Y-u_Y\cdot u_X$.  In the rest, for $X\in {\rm rep}_K(Q,I)$, we use $X$ instead of $u_X$ for simplicity.

\smallskip
Recall that there is a symmetric bilinear form $(-,-)_A$ on the Grothendieck group $K_0(\mathrm{rep}_K(Q,I))$ with integer values. For a representation $M\in\mathrm{rep}_K(Q,I)$ denote its image in $K_0(\mathrm{rep}_K(Q,I))$ by $h_M$. Let $\tilde{L}(A)=K_0(\mathrm{mod} A)\oplus L(A)$ and extend the Lie bracket on $L(A)$ by
\[
[h_M,h_N]=0,~ [h_M,N]=(M,N)_A N=-[N,h_M]
\]
for $M,N\in\mathrm{rep}_K(Q,I)$. Then $\tilde{L}(A)$ becomes a Lie algebra.

\subsection{The main result}

\medskip
Let $\mathfrak{n}$ be the $\mathbb{Z}$-submodule of $\mathfrak{g}$ spanned by the basis elements in Proposition~\ref{prop:basis} and $\mathfrak{n}^+$ the one spanned by the elements in (g1). They are Lie subalgebras of $\mathfrak{g}$, $\mathfrak{n}\otimes_{\mathbb{Z}}\mathbb{C}=\mathfrak{g}$ and $\mathfrak{n}^+\otimes_{\mathbb{Z}}\mathbb{C}=\mathfrak{g}^+$. The following is the main result of this paper. Let $S_1=W_{1,1},\ldots,S_{n-1}=W_{n-1,n-1},S_n=V_n,S'_n=U_{n,n}$.

\begin{theorem}
\label{thm:ringel-hall-lie-algebra}
\begin{itemize}
\item[\rm (a)]  There is an injective homomorphism $\mathfrak{n}^+\rightarrow L(A)$ of Lie algebras over $\mathbb{Z}$ which sends $x_1, \ldots, x_{n-1}, x_n, x'_n$ to $S_1,\ldots,S_{n-1},S_n,S'_n$. By extending the scalars, it becomes an isomorphism $\mathfrak{n}^+\otimes_{\mathbb{Z}}\mathbb{Z}[\frac{1}{2}]\overset{\cong}\longrightarrow L(A)\otimes_{\mathbb{Z}}\mathbb{Z}[\frac{1}{2}]$ of Lie algebras over $\mathbb{Z}[\frac{1}{2}]$.
\item[\rm (b)] The homomorphism $\mathfrak{n}^+\rightarrow L(A)$ in (a) extends to an injective homomorphism $\mathcal{U}(\mathfrak{n}^+)\rightarrow\mathcal{H}(A)$ of $\mathbb{Z}$-algebras. By extending the scalars, it becomes an isomorphism $\mathcal{U}(\mathfrak{n}^+)\otimes_{\mathbb{Z}}\mathbb{Q}\overset{\cong}\longrightarrow \mathcal{H}(A)\otimes_{\mathbb{Z}}\mathbb{Q}$ of  $\mathbb{Q}$-algebras.
\item[\rm (c)] The homomorphism $\mathfrak{n}^+\longrightarrow L(A)$  in (a) extends to an injective homomorphism $\varphi\colon\mathfrak{n}\longrightarrow \tilde{L}(A)$ of Lie algebras over $\mathbb{Z}$, by $h_i\mapsto h_{S_i}$ ($1\leq i\leq n-1$) and $h_n\mapsto 2h_{S_n}$. By extending the scalars, it becomes an isomorphism $\mathfrak{n}\otimes_{\mathbb{Z}}\mathbb{Z}[\frac{1}{2}]\overset{\cong}\longrightarrow \tilde{L}(A)\otimes_{\mathbb{Z}}\mathbb{Z}[\frac{1}{2}]$ of Lie algebras over $\mathbb{Z}[\frac{1}{2}]$.
\end{itemize}
\end{theorem}

To prove Theorem~\ref{thm:ringel-hall-lie-algebra} we need some preparation. First we describe the Lie brackets of the basis elements of $L(A)$.

\begin{proposition}
\label{rel}
\label{prop:lie-brackets-of-L(A)}
The following equalities hold in $L(A)$:
\begin{itemize}
         \item[\rm (a)] $[W_{i,j}, W_{l,m}]=\delta_{j+1,l}W_{i,m}-\delta_{m+1,i}W_{l,j}$ for $1\leq i\leq j\leq n-1$ and $1\leq l\leq m\leq n-1$,
         \item[\rm (b)] $[W_{i,j}, V_l]=\delta_{j+1,l}V_i$ for $1\leq i\leq j\leq n-1$ and $1\leq l\leq n$,
         \item[\rm (c)] $[W_{i,j}, U_{l,m}]=\delta_{j+1,m}U_{l,i}+\delta_{j+1,l}U_{i,m}$ for $1\leq i\leq j\leq n$ and $1\leq l,m\leq n$,
         \item[\rm (d)] $[V_i, V_j]=U_{j,i}-U_{i,j}$ for $1\leq i,j\leq n$,
\end{itemize}
\noindent where $\delta_{i,j}$ is the Kronecker symbol. For other pairs of indecomposable representations, the Lie bracket of them equals 0.
\end{proposition}
\begin{proof}
We only prove (c) and the proof of the other cases is similar. It is clear that $\dimv W_{i,j}+\dimv U_{l,m}=\varepsilon_i-\varepsilon_{j+1}+\varepsilon_l+\varepsilon_m$ belongs to $\Phi^+$ if and only if $l=j+1$ or $m=j+1$. So $[W_{i,j},U_{l,m}]=0$ unless  $l=j+1$ or $m=j+1$, by Theorem~\ref{thm:Gabriel's-theorem}(a).

Case 1: $l=m=j+1$. Then
 $\dimv W_{i,j}+\dimv U_{j+1,j+1}=\varepsilon_i+\varepsilon_{j+1}$, so by Theorem~\ref{thm:Gabriel's-theorem}(b) the element $[W_{i,j},U_{j+1,j+1}]$ is a linear combination of $U_{i,j+1}$ and $U_{j+1,i}$.
Now it is easy to see from the structures of $W_{i,j}$, $U_{j+1,j+1}$, $U_{i,j+1}$ and  $U_{j+1,i}$ in Section~\ref{s:representation-theory} that neither $U_{i,j+1}$ nor $U_{j+1,i}$ has a submodule isomorphic to $W_{i,j}$, and both $U_{i,j+1}$ and $U_{j+1,i}$ have a unique submodule isomorphic to $U_{j+1,j+1}$ with quotient  isomorphic to $W_{i,j}$. So $[W_{i,j}, U_{j+1,j+1}]=U_{i,j+1}+U_{j+1,i}$.

Case 2: $l=j+1\neq m$.  Then $\dimv W_{i,j}+\dimv U_{j+1,m}=\varepsilon_i+\varepsilon_{m}$, so by Theorem~\ref{thm:Gabriel's-theorem}(b) the element $[W_{i,j},U_{j+1,m}]$ is a linear combination of $U_{i,m}$ and $U_{m,i}$. Now it is easy to see from the structures of $W_{i,j}$, $U_{j+1,m}$, $U_{i,m}$ and  $U_{m,i}$ in Section~\ref{s:representation-theory} that neither $U_{i,m}$ nor $U_{m,i}$ has a submodule isomorphic to $W_{i,j}$, $U_{i,m}$ has a unique submodule isomorphic to $U_{j+1,m}$ with quotient  isomorphic to $W_{i,j}$, and $U_{m,i}$ has no submodule isomorphic to $U_{j+1,m}$. So $[W_{i,j}, U_{j+1,m}]=U_{i,m}$.

Case 3: $l\neq m=j+1$. This is similar to Case 2.
\end{proof}

\begin{corollary}
\label{cor:generators-of-L(A)}
  $\{S_1,\ldots,S_{n-1},S_n,S'_n\}$ is a set of generators of $L(A)\otimes_{\mathbb{Z}}\mathbb{Z}[\frac{1}{2}]$.
\end{corollary}
\begin{proof}
We need to prove that any indecomposable representation $M$ is generated by $W_{1,1}, W_{2,2}, \ldots W_{n-1,n-1}, V_n, U_{n,n}$.

Case 1: $M=W_{i,j}$, $1\leq i\leq j\leq n-1$. This follows from $[W_{i,i},W_{i+1,j}]\overset{\ref{rel}\rm(a)}=W_{i,j}$  by induction on $j$.

Case 2: $M=V_i$, $1\leq i\leq n$. This follows from  $[W_{i,i}, V_{i+1}]\overset{\ref{rel}\rm(b)}=V_{i}$ by decreasing induction on $i$.

Case 3: $M=U_{i,j}$, $1\leq i,j\leq n$. We first assume $i=n$. Then $[W_{j,n-1},U_{n,n}]\overset{\ref{rel}\rm(c)}=U_{n,j}+U_{j,n}$. Moreover, $[V_j,V_n]\overset{\ref{rel}\rm(d)}=U_{n,j}-U_{j,n}$. So $U_{j,n}=\frac{[W_{j,n-1},U_{n,n}]-[V_j,V_n]}{2}$ and $U_{n,j}=\frac{[W_{j,n-1},U_{n,n}]+[V_j,V_n]}{2}$. This also solves the problem for $j=n$. Next we assume $i\neq n$ and $j\neq n$. Then $U_{i,j}\overset{\ref{rel}\rm(c)}=[W_{i,n-1},U_{n,j}]$. The proof finishes.
\end{proof}

Now we are ready to prove Theorem~\ref{thm:ringel-hall-lie-algebra}.

\begin{proof}[Proof of Theorem~\ref{thm:ringel-hall-lie-algebra}.]
(a) By Proposition~\ref{prop:lie-brackets-of-L(A)}, the assignment $x_1\mapsto S_1,\ldots,x_{n-1}\mapsto S_{n-1}, x_n\mapsto S_n, x'_n\mapsto S'_n$ extends to a homomorphism $\mathfrak{n}^+\rightarrow L(A)$ of Lie algebras over $\mathbb{Z}$. By extending the scalars, we obtain a homomorphism $\mathfrak{n}^+\otimes_{\mathbb{Z}}\mathbb{Z}[\frac{1}{2}]\rightarrow L(A)\otimes_{\mathbb{Z}}\mathbb{Z}[\frac{1}{2}]$ of Lie algebras over $\mathbb{Z}[\frac{1}{2}]$. By Corollary~\ref{cor:generators-of-L(A)}, this homomorphism is surjective, and hence bijective because both sides are free over $\mathbb{Z}[\frac{1}{2}]$ of rank $\frac{3n^2+n}{2}$. This implies that the homomorphism $\mathfrak{n}^+\rightarrow L(A)$ is injective.

(b) We have the following chain of homomorphisms of $\mathbb{Z}$-algebras:
\[
\mathcal{U}(\mathfrak{n}^+)\longrightarrow \mathcal{U}(L(A))\longrightarrow \mathcal{H}(A).
\]
The first homomorphism is injective by (a), and the second one is injective by \cite[Proposition 3.1]{Riedtmann94}. So the composition is injective. By extending the scalars we obtain a chain of injective homomorphisms of $\mathbb{Q}$-algebras:
\[
\mathcal{U}(\mathfrak{n}^+)\otimes_{\mathbb{Z}}\mathbb{Q}\longrightarrow \mathcal{U}(L(A))\otimes_{\mathbb{Z}}\mathbb{Q}\longrightarrow \mathcal{H}(A)\otimes_{\mathbb{Z}}\mathbb{Q}.
\]
The first homomorphism is bijective by (a) and the second one is bijective by the proof of \cite[Proposition 3.1]{Riedtmann94}.

(c) It is enough to prove $\varphi([h_i,x_j])=[\varphi(h_i),\varphi(x_j)]$ and $\varphi([h_i,x'_n])=[\varphi(h_i),\varphi(x'_n)]$, more precisely,
\begin{itemize}
\item[-] $\varphi([h_i,x_j])=[h_{S_i},S_j]=(S_i,S_j)_A S_j$ for $1\leq i\leq n-1$ and $1\leq j\leq n$;
\item[-] $\varphi([h_i,x'_n])=[h_{S_i},S'_n]=(S_i,S'_n)S'_n$ for $1\leq i\leq n-1$;
\item[-] $\varphi([h_n,x_j])=[2h_{S_n},S_j]=2(S_n,S_j)_A S_j$ for $1\leq j\leq n$;
\item[-] $\varphi([h_n,x'_n])=[2h_{S_n},S'_n]=2(S_n,S'_n)S'_n$.
\end{itemize}
This is straightforward.
\end{proof}

\begin{remark} 
\label{rem:correspondance-on-basis}
On basis elements, the homomorphism in Theorem~\ref{thm:ringel-hall-lie-algebra}(a) is given by
\begin{align*}
x_{i,j-1}&\mapsto W_{i,j-1}~~(1\leq i< j\leq n)\\
x_{i,n}&\mapsto V_i~~(1\leq i\leq n)\\
[x_{i,n},x_{j,n}] &\mapsto U_{j,i}-U_{i,j}~~(1\leq i<j\leq n)\\
x'_n&\mapsto U_{n,n}\\
x'_{i,n}&\mapsto U_{i,n}+U_{n,i}~~(1\leq i\leq n-1)\\
[x_{j,n-1},x'_{i,n}]&\mapsto U_{i,j}+U_{j,i}~~(1\leq i\leq j\leq n-1).
\end{align*}
\end{remark}

We have the following corollary of Theorem~\ref{thm:ringel-hall-lie-algebra}, Corollary\ref{cor:type-b-and-c-as-quotient} and Remark~\ref{rem:correspondance-on-basis}.
\begin{corollary}
The subspace of $\tilde{L}(A)\otimes_{\mathbb{Z}}\mathbb{C}$ spanned by $U_{i,j}+U_{j,i}~(1\leq i\leq j\leq n)$ is a Lie ideal and the corresponding quotient is isomorphic to $\mathfrak{b}_{\mathbf{B}}$. The subspace of $\tilde{L}(A)\otimes_{\mathbb{Z}}\mathbb{C}$ spanned by $V_i~(1\leq i\leq n),~U_{j,i}-U_{i,j}~(1\leq i< j\leq n)$ is a Lie ideal and the corresponding quotient is isomorphic to $\mathfrak{b}_{\mathbf{C}}$.
\end{corollary}



\section{Appendix}
In this appendix we show that the bilinear form $(-,-)_A$ introduced in Section~\ref{s:representation-theory} is a modified symmetric Euler form for $A$.

\medskip

Since the global dimension of the algebra $A$ is infinite, the usual Euler form $\langle M, N\rangle=\sum_{p=0}^{\infty}(-1)^{p}\dim_K {\rm Ext}_A^p(M,N)$ is not well-defined. We modify it as follows.
First define 
\[
\langle  M, N\rangle_t=\sum_{p=0}^{\infty}\dim_K{\rm Ext}^p(M,N)(-t)^p.
\] 
This form is not additive but additive with respect to split short exact sequences.

\begin{theorem}
\label{thm:euler-form-vs-cartan-matrix}
\begin{itemize}
\item[\rm (a)] For any $M$ and $N$ in $\mathrm{rep}_K(Q,I)$, the power series $\langle M,N\rangle_t$ is a rational function in $t$ and $t=1$ is not a pole.
\item[\rm (b)] The form $\langle -,-\rangle_1$ is additive.
\item[\rm (c)] The matrix of the form $\langle -,-\rangle_1$ with respect to the basis $\{S_1,\ldots,S_n\}$ is $C_A^{-t}$.
\end{itemize}
\end{theorem}

By Theorem~\ref{thm:euler-form-vs-cartan-matrix}(b), the form $\langle-,-\rangle_1$ can be considered as a bilinear form on $K_0(\mathrm{rep}_K(Q,I))$. Restricted to the subgroup generated by the classes of $S_1,\ldots,S_{n-1},U_{n,n}$, it is exactly the Euler form $\langle-,-\rangle_H$ in \cite[Section 3.3]{GeissLeclercSchroeer16} because these modules have projective dimension $1$.
Recall from Definition~\ref{def:bilinear-form} that there is a bilinear form $(-,-)_A$.

\begin{corollary}
For any $M$ and $N$ in ${\rm rep}_K(Q,I)$ we have
\[
(M,N)_A=\langle M,N\rangle_1+\langle N,M\rangle_1
\]
\end{corollary}
\begin{proof}
By Theorem~\ref{thm:euler-form-vs-cartan-matrix}(c), the matrix of the symmetric form associated to $\langle-,-\rangle_1$ is exactly $C_A^{-1}+C_A^{-t}$.
\end{proof}

We start to prove Theorem~\ref{thm:euler-form-vs-cartan-matrix}. We first make a table for $\langle M,N\rangle_t$, where $M$ and $N$ are indecomposable. 
This is obtained by a direct computation on ${\rm Ext}_A^*(M,N)$, which we omit.

\begin{lemma}\label{NEC}
\label{lem:table-for-euler-form-in-t}
The following Table 1 gives the values of $\langle M,  N\rangle_t$ with column $M$ and row $N$.

\begin{table}
\label{tab:euler-form-in-t}
\caption{The values of $\langle M,  N\rangle_t$}
\medskip
 \resizebox{\textwidth}{65mm}{
\begin{tabular}{|c|c|c|c|}
  \hline
   & $U_{i,j}$ & $V_i$ & $W_{i,j}$  \\ \hline
   $U_{l,k}(1\leq k\leq l<n)$ & $\left\{
                                   \begin{array}{ll}
                                     -2t, & \hbox{$l< \min\{i,j\}$;} \\
                                     -t, & \hbox{$k<j\leq l<i$;} \\
                                     -t, & \hbox{$k<i\leq l<j$;} \\
                                     0, & \hbox{$k<i\leq j\leq l$;} \\
                                     0, & \hbox{$k<j<i\leq l$;} \\
                                     0, & \hbox{$j\leq k\leq l<i$;} \\
                                     1, & \hbox{$j\leq k<i \leq l$;} \\
                                     1-t, & \hbox{$i\leq k\leq l<j$;} \\
                                     1, & \hbox{$i\leq k<j \leq l$;} \\
                                     2, & \hbox{$\max\{i,j\}\leq k$.}
                                   \end{array}
                                 \right.$
 & $\left\{
                                   \begin{array}{ll}
                                     -t, & \hbox{$k\leq l<i$;} \\
                                     0, & \hbox{$k<i\leq l$;} \\
                                     1, & \hbox{$i\leq k\leq l$.}
                                   \end{array}
                                 \right.$ &  $\left\{
                                   \begin{array}{ll}
                                     1, & \hbox{$k<i\leq l\leq j$;} \\
                                     2, & \hbox{$i\leq k\leq l\leq j$;} \\
                                     1, & \hbox{$i\leq k\leq j < l$;} \\
                                     0, & \hbox{other.}
                                   \end{array}
                                 \right.$   \\ \hline
  $U_{l,k}(1\leq l<k\leq n)$& $\left\{
                                   \begin{array}{ll}
                                     -2t, & \hbox{$k< \min\{i,j\}$;} \\
                                     -t, & \hbox{$l<j\leq k<i$;} \\
                                     1-2t, & \hbox{$l<i\leq k<j$;} \\
                                     1-t, & \hbox{$l<i\leq j\leq k$;} \\
                                     1-t, & \hbox{$l<j<i\leq k$;} \\
                                     0, & \hbox{$j\leq l<k<i$;} \\
                                     1, & \hbox{$j\leq l<i \leq k$;} \\
                                     1-t, & \hbox{$i\leq l<k<j$;} \\
                                     2-t, & \hbox{$i\leq l<j \leq k$;} \\
                                     2, & \hbox{$\max\{i,j\}\leq l$.}
                                   \end{array}
                                 \right.$ & $\left\{
                                   \begin{array}{ll}
                                     -t, & \hbox{$l<k<i$;} \\
                                     1-t, & \hbox{$l<i\leq k$;} \\
                                     1, & \hbox{$i\leq l<k$.}
                                   \end{array}
                                 \right.$ &  $\left\{
                                   \begin{array}{ll}
                                     1, & \hbox{$l<i\leq k\leq j$;} \\
                                     2, & \hbox{$i\leq l<k\leq j$;} \\
                                     1, & \hbox{$i\leq l\leq j < k$;} \\
                                     0, & \hbox{other.}
                                   \end{array}
                                 \right.$      \\ \hline
  $U_{l,k}(1\leq k\leq l=n)$ & $\left\{
                                   \begin{array}{ll}
                                     0, & \hbox{$k<\min\{i,j\}$;} \\
                                     1, & \hbox{$j\leq k<i$;} \\
                                     1, & \hbox{$i\leq k<j$;} \\
                                     2, & \hbox{$\max\{i,j\}\leq k$.}
                                   \end{array}
                                 \right.$ & $\left\{
                                   \begin{array}{ll}
                                     0, & \hbox{$k<i$;} \\
                                     1, & \hbox{$i\leq k$.}
                                   \end{array}
                                 \right.$ & $\left\{
                                   \begin{array}{ll}
                                     1, & \hbox{$i\leq k\leq j$;} \\
                                     0, & \hbox{other.}
                                   \end{array}
                                 \right.$      \\ \hline
  $V_l(1\leq l\leq n)$ & $\left\{
                                   \begin{array}{ll}
                                     -t, & \hbox{$l<\min\{i,j\}$;} \\
                                     0, & \hbox{$j\leq l<i$;} \\
                                     1-t, & \hbox{$i\leq l<j$;} \\
                                     1, & \hbox{$\max\{i,j\}\leq l$.}
                                   \end{array}
                                 \right.$ & $\left\{
                                   \begin{array}{ll}
                                     \sum_{p=0}^{\infty}(-t)^p, & \hbox{$i\leq l$;} \\
                                     \sum_{p=1}^{\infty}(-t)^p, & \hbox{$l< i$.}
                                   \end{array}
                                 \right.$ &  $\left\{
                                   \begin{array}{ll}
                                     1, & \hbox{$i\leq l\leq j$;} \\
                                     0, & \hbox{other.}
                                   \end{array}
                                 \right.$        \\ \hline
  $W_{l,k}(1\leq l\leq k< n)$ & $\left\{
                                   \begin{array}{ll}
                                     0, & \hbox{$l<k+1<\min\{i,j\}$;} \\
                                     -t, & \hbox{$l<j\leq k+1<i$;} \\
                                     -t, & \hbox{$l<i\leq k+1<j$;} \\
                                     -2t, & \hbox{$l<i\leq j\leq k+1$;} \\
                                     -2t, & \hbox{$l<j<i\leq k+1$;} \\
                                     0, & \hbox{$j\leq l<k+1< i$;} \\
                                     -t, & \hbox{$j\leq l<i \leq k+1$;} \\
                                     0, & \hbox{$i\leq l<k+1<j$;} \\
                                     -t, & \hbox{$i\leq l<j \leq k+1$;} \\
                                     0, & \hbox{$\max\{i,j\}\leq l$.}
                                   \end{array}
                                 \right.$ & $\left\{
                                   \begin{array}{ll}
                                     0, & \hbox{$l<k+1<i$;} \\
                                     -t, & \hbox{$l<i\leq k+1$;} \\
                                     0, & \hbox{$i\leq l<k+1$.}
                                   \end{array}
                                 \right.$ & $\left\{
                                   \begin{array}{ll}
                                     1, & \hbox{$i\leq l\leq j<k+1$;} \\
                                     -t, & \hbox{$l<i\leq k+1\leq j$;} \\
                                     0, & \hbox{$i\leq l<k+1\leq j$;} \\
                                     0, & \hbox{other.}
                                   \end{array}
                                 \right.$     \\
  \hline
\end{tabular}}
\end{table}
\end{lemma}

Notice that in the table of Lemma~\ref{lem:table-for-euler-form-in-t} all the entries are polynomials in $t$ but two, which are $\sum_{p=0}^{\infty}(-t)^p=\frac{1}{1+t}$. So Theorem~\ref{thm:euler-form-vs-cartan-matrix}(a) is proved and the form $\langle  M,  N\rangle_1$ is well-defined.

\begin{corollary}\label{Char}
\label{cor:table-for-euler-form}
The following Table 2 gives the values of $\langle M, N\rangle_1$ with column $M$ and row $N$.
\begin{table}
\caption{The values of $\langle M,  N\rangle_1$}
\medskip
 \resizebox{\textwidth}{65mm}{
\begin{tabular}{|c|c|c|c|}
  \hline
   & $U_{i,j}$ & $V_i$ & $W_{i,j}$ \\  \hline
   $U_{l,k}(1\leq k\leq l<n)$ & $\left\{
                                   \begin{array}{ll}
                                     -2, & \hbox{$l< \min\{i,j\}$;} \\
                                     -1, & \hbox{$k<j\leq l<i$;} \\
                                     -1, & \hbox{$k<i\leq l<j$;} \\
                                     0, & \hbox{$k<i\leq j\leq l$;} \\
                                     0, & \hbox{$k<j<i\leq l$;} \\
                                     0, & \hbox{$j\leq k\leq l<i$;} \\
                                     1, & \hbox{$j\leq k<i \leq l$;} \\
                                     0, & \hbox{$i\leq k\leq l<j$;} \\
                                     1, & \hbox{$i\leq k<j \leq l$;} \\
                                     2, & \hbox{$\max\{i,j\}\leq k$.}
                                   \end{array}
                                 \right.$
 & $\left\{
                                   \begin{array}{ll}
                                     -1, & \hbox{$k\leq l<i$;} \\
                                     0, & \hbox{$k<i\leq l$;} \\
                                     1, & \hbox{$i\leq k\leq l$.}
                                   \end{array}
                                 \right.$ &  $\left\{
                                   \begin{array}{ll}
                                     1, & \hbox{$k<i\leq l\leq j$;} \\
                                     2, & \hbox{$i\leq k\leq l\leq j$;} \\
                                     1, & \hbox{$i\leq k\leq j < l$;} \\
                                     0, & \hbox{other.}
                                   \end{array}
                                 \right.$   \\ \hline
  $U_{l,k}(1\leq l<k\leq n)$& $\left\{
                                   \begin{array}{ll}
                                     -2, & \hbox{$k< \min\{i,j\}$;} \\
                                     -1, & \hbox{$l<j\leq k<i$;} \\
                                     -1, & \hbox{$l<i\leq k<j$;} \\
                                     0, & \hbox{$l<i\leq j\leq k$;} \\
                                     0, & \hbox{$l<j<i\leq k$;} \\
                                     0, & \hbox{$j\leq l<k<i$;} \\
                                     1, & \hbox{$j\leq l<i \leq k$;} \\
                                     0, & \hbox{$i\leq l<k<j$;} \\
                                     1, & \hbox{$i\leq l<j \leq k$;} \\
                                     2, & \hbox{$\max\{i,j\}\leq l$.}
                                   \end{array}
                                 \right.$ & $\left\{
                                   \begin{array}{ll}
                                     -1, & \hbox{$l<k<i$;} \\
                                     0, & \hbox{$l<i\leq k$;} \\
                                     1, & \hbox{$i\leq l<k$.}
                                   \end{array}
                                 \right.$ &  $\left\{
                                   \begin{array}{ll}
                                     1, & \hbox{$l<i\leq k\leq j$;} \\
                                     2, & \hbox{$i\leq l<k\leq j$;} \\
                                     1, & \hbox{$i\leq l\leq j < k$;} \\
                                     0, & \hbox{other.}
                                   \end{array}
                                 \right.$      \\ \hline
  $U_{l,k}(1\leq k\leq l=n)$ & $\left\{
                                   \begin{array}{ll}
                                     0, & \hbox{$k<\min\{i,j\}$;} \\
                                     1, & \hbox{$j\leq k<i$;} \\
                                     1, & \hbox{$i\leq k<j$;} \\
                                     2, & \hbox{$\max\{i,j\}\leq k$.}
                                   \end{array}
                                 \right.$ & $\left\{
                                   \begin{array}{ll}
                                     0, & \hbox{$k<i$;} \\
                                     1, & \hbox{$i\leq k$.}
                                   \end{array}
                                 \right.$ & $\left\{
                                   \begin{array}{ll}
                                     1, & \hbox{$i\leq k\leq j$;} \\
                                     0, & \hbox{other.}
                                   \end{array}
                                 \right.$      \\ \hline
  $V_l(1\leq l\leq n)$ & $\left\{
                                   \begin{array}{ll}
                                     -1, & \hbox{$l<\min\{i,j\}$;} \\
                                     0, & \hbox{$j\leq l<i$;} \\
                                     0, & \hbox{$i\leq l<j$;} \\
                                     1, & \hbox{$\max\{i,j\}\leq l$.}
                                   \end{array}
                                 \right.$ & $\left\{
                                   \begin{array}{ll}
                                     \frac{1}{2}, & \hbox{$i\leq l$;} \\
                                     -\frac{1}{2}, & \hbox{$l< i$.}
                                   \end{array}
                                 \right.$ &  $\left\{
                                   \begin{array}{ll}
                                     1, & \hbox{$i\leq l\leq j$;} \\
                                     0, & \hbox{other.}
                                   \end{array}
                                 \right.$        \\ \hline
  $W_{l,k}(1\leq l\leq k< n)$ & $\left\{
                                   \begin{array}{ll}
                                     0, & \hbox{$l<k+1<\min\{i,j\}$;} \\
                                     -1, & \hbox{$l<j\leq k+1<i$;} \\
                                     -1, & \hbox{$l<i\leq k+1<j$;} \\
                                     -2, & \hbox{$l<i\leq j\leq k+1$;} \\
                                     -2, & \hbox{$l<j<i\leq k+1$;} \\
                                     0, & \hbox{$j\leq l<k+1< i$;} \\
                                     -1, & \hbox{$j\leq l<i \leq k+1$;} \\
                                     0, & \hbox{$i\leq l<k+1<j$;} \\
                                     -1, & \hbox{$i\leq l<j \leq k+1$;} \\
                                     0, & \hbox{$\max\{i,j\}\leq l$.}
                                   \end{array}
                                 \right.$ & $\left\{
                                   \begin{array}{ll}
                                     0, & \hbox{$l<k+1<i$;} \\
                                     -1, & \hbox{$l<i\leq k+1$;} \\
                                     0, & \hbox{$i\leq l<k+1$.}
                                   \end{array}
                                 \right.$ & $\left\{
                                   \begin{array}{ll}
                                     1, & \hbox{$i\leq l\leq j<k+1$;} \\
                                     -1, & \hbox{$l<i\leq k+1\leq j$;} \\
                                     0, & \hbox{$i\leq l<k+1\leq j$;} \\
                                     0, & \hbox{other.}
                                   \end{array}
                                 \right.$     \\
  \hline
\end{tabular}}
\end{table}
\end{corollary}

The following proposition is Theorem~\ref{thm:euler-form-vs-cartan-matrix}(b).

\begin{proposition}
The form $\langle -, -\rangle_1$ is additive, i.e., $\forall M$ and $N$ in ${\rm rep}_K(Q,I)$,
 $$\langle M, N\rangle_1=\sum_{i=1}^{n}\sum_{j=1}^{n}({\rm dim}_KM_i)({\rm dim}_KN_j)\langle S_i, S_j\rangle_1.$$
\end{proposition}
\begin{proof}
We may assume that $M$ and $N$ are indecomposable.
We prove the equality for some typical cases by explicitly computing the values of both sides.

\smallskip

Case 1. $M=U_{l,k}$ and $N=U_{i,j}$ for $i \leq k \leq l < j$. By Corollary \ref{Char} we have
\[{\rm LHS}=\langle U_{l,k}, U_{i,j}\rangle_1=0,\]
and by direct computation,
\[\begin{split}
  {\rm RHS}=&\sum_{a=k}^{n}\sum_{a'=i}^{n}\langle S_a, S_{a'}\rangle_1+\sum_{a=k}^{n}\sum_{a'=j}^{n}\langle S_a, S_{a'}\rangle_1 +\sum_{a=l}^{n}\sum_{a'=i}^{n}\langle S_a, S_{a'}\rangle_1+ \sum_{a=l}^{n}\sum_{a'=j}^{n}\langle S_a, S_{a'}\rangle_1\\
=&\sum_{a=k}^{n}\langle S_a, S_a\rangle+\sum_{a=k}^{n-1}\langle S_a, S_{a+1}\rangle_1+\sum_{a=j}^{n}\langle S_a, S_a\rangle_1+\sum_{a=j}^{n}\langle S_{a-1}, S_a\rangle_1\\
&+\sum_{a=l}^{n}\langle S_a, S_a\rangle_1+\sum_{a=l}^{n-1}\langle S_a, S_{a+1}\rangle_1+\sum_{a=j}^{n}\langle S_a, S_a\rangle_1+\sum_{a=j}^{n}\langle S_{a-1}, S_a\rangle_1\\
=&[(n-k)\times 1+\frac{1}{2}]+(n-k)\times(-1)+[(n-j)\times 1+\frac{1}{2}]+(n-j+1)\times(-1)\\
&+[(n-l)\times 1+\frac{1}{2}]+(n-l)\times(-1)+[(n-j)\times 1+\frac{1}{2}]+(n-j+1)\times(-1)\\
=&\frac{1}{2}+\frac{1}{2}-1+\frac{1}{2}+\frac{1}{2}-1=0.\\
\end{split}
\]

Case 2. $M=V_{l}$ and $N=V_{i}$ for $l<i$. By Corollary \ref{Char}
\[{\rm LHS}=\langle V_{l}, V_{i}\rangle_1=-\frac{1}{2},\]
and on the other hand,
\[\begin{split}
{\rm RHS}
=&\sum_{a=l}^{n}\sum_{a'=i}^{n}\langle S_a, S_{a'}\rangle_1=\sum_{a=i}^{n}\langle S_a, S_a\rangle_1+\sum_{a=i}^{n}\langle S_{a-1}, S_a\rangle_1\\
=&[(n-i)\times 1+\frac{1}{2}]+(n-i+1)\times(-1)=\frac{1}{2}-1
=-\frac{1}{2}.
\end{split}
\]

Case 3. $M=W_{l,k}$ and $N=W_{i,j}$ for $l < i\leq k + 1 \leq j(<n)$. By Corollary \ref{Char}
\[{\rm LHS}=\langle W_{l,k}, W_{i,j}\rangle=-1,\]
and by direct computation,
\[\begin{split}
{\rm RHS}
=&\sum_{a=l}^{k}\sum_{a'=i}^{j}\langle S_a, S_{a'}\rangle_1\\
=&\sum_{a=i}^{k}\langle S_a, S_a\rangle_1+\sum_{a=i}^{k}\langle S_{a-1}, S_a\rangle_1+\langle S_k, S_{k+1}\rangle_1 \quad(\text{note: } k+1<n)\\
=&(k-i+1)\times 1+(k-i+1)\times(-1)+(-1)=-1.
\end{split}
\]

The proof of the other cases is similar and we obtain the additivity.
\end{proof}

Finally, by Corollary~\ref{cor:table-for-euler-form} the matrix of the form $\langle M,N\rangle_1$ with respect to the basis $\{S_1,\ldots,S_n\}$ is:
\[\left(
  \begin{array}{ccccccc}
    1& -1 & 0 & \ldots & 0 & 0 & 0 \\
    0& 1 & -1 &\ldots & 0 & 0 & 0 \\
    0& 0 & 1 & \ldots &0 & 0 & 0 \\
    \vdots& \vdots & \vdots & \ddots & \vdots & \vdots & \vdots \\
    0 & 0 & 0 & \ldots & 1 &-1 & 0 \\
    0 & 0 & 0 & \ldots & 0 & 1 & -1 \\
    0 & 0 & 0 & \ldots & 0 & 0 & \frac{1}{2} \\
  \end{array}
\right),\]
which is exactly $C_A^{-t}$. This is Theorem~\ref{thm:euler-form-vs-cartan-matrix}(c).


this study.

\def\cprime{$'$}
\providecommand{\bysame}{\leavevmode\hbox to3em{\hrulefill}\thinspace}
\providecommand{\MR}{\relax\ifhmode\unskip\space\fi MR }
\providecommand{\MRhref}[2]{%
  \href{http://www.ams.org/mathscinet-getitem?mr=#1}{#2}
}
\providecommand{\href}[2]{#2}

\end{document}